\documentclass[12pt]{amsart}
\usepackage{amssymb}
\usepackage{enumerate}

\newcommand{\N}{{\mathbb N}}
\newcommand{\C}{{\mathbb C}}
\newcommand{\Z}{{\mathbb Z}}

\newcommand{\wand}{wandering domain}

\newcommand{\tef}{transcendental entire function}

\newcommand{\nhd}{neighbourhood}
\newcommand{\sconn}{simply connected}
\newcommand{\mconn}{multiply connected}

\newcommand{\sw}{spider's web}

\theoremstyle{plain}
\newtheorem{theorem}{Theorem}[section]

\newtheorem*{theorem*}{Theorem}
\newtheorem*{definition*}{Definition}
\newtheorem*{proposition*}{Proposition}

\newtheorem{lemma}[theorem]{Lemma}
\theoremstyle{definition}

\theoremstyle{remark}

\theoremstyle{problem}

\theoremstyle{example}

\newenvironment{myindentpar}[1]%
{\begin{list}{}%
         {\setlength{\leftmargin}{#1}}%
         \item[]%
}
{\end{list}}

\begin{document}


\title[The structure of spider's web fast escaping sets]{The structure of spider's web\\ fast escaping sets}

\author{J.W. Osborne}
\address{Department of Mathematics and Statistics \\
   The Open University \\
   Walton Hall\\
   Milton Keynes MK7 6AA\\
   UK}
\email{j.osborne@open.ac.uk}





\begin{abstract}
Building on recent work by Rippon and Stallard, we explore the intricate structure of the spider's web fast escaping sets associated with certain \tef s. Our results are expressed in terms of the components of the complement of the set (the `holes' in the web).  We describe the topology of such components and give a characterisation of their possible orbits under iteration.  We show that there are uncountably many components having each of a number of orbit types, and we prove that components with bounded orbits are quasiconformally homeomorphic to components of the filled Julia set of a polynomial.  We also show that there are singleton periodic components and that these are dense in the Julia set.    
\end{abstract}

\maketitle

\section{Introduction}
\label{intro}
\setcounter{equation}{0} 

For a \tef\ $ f $, we denote the $ n $th iterate of $ f $ for $ n \in \N $ by $ f^n $.  The Fatou set, $ F(f), $ is the set of points, $ z \in \C, $ such that the family of functions $ \lbrace f^n : n \in \N \rbrace $ is normal in some \nhd\ of $ z $, and the Julia set, $ J(f) $, is the complement of $ F(f) $.  For an introduction to the properties of these sets and the iteration theory of \tef s, we refer to \cite{wB93}.

This paper is concerned with the \textit{escaping set} of $ f $, defined by:
\[ I(f) = \lbrace z \in \C : f^n(z) \to \infty \text{ as } n \to \infty \rbrace, \]
and a subset of $ I(f), $ known as the \textit{fast escaping set}, defined as follows:
\[ A(f) = \lbrace z \in \C : \text{there exists } L \in \N \text{ such that } \vert f^{n+L}(z) \vert \geq M^n(R,f), \text{ for } n \in \N \rbrace.  \]
Here, and throughout the paper,
\[ M(r, f) = \max_{\vert z \vert = r} \vert f(z) \vert, \text{ for } r > 0, \]
while $ M^n(r,f) $ denotes the $ n $th iterate of $ M $ with respect to $ r $, and $ R > 0 $ is chosen so that $ M(r, f) > r $ for $ r \geq R $.  For brevity, we do not repeat this restriction on $ R $ except in formal statements of results, but it should always be assumed to apply.  We frequently abbreviate $ M(r,f) $ to $ M(r) $ where $ f $ is clear from the context.

The escaping set, $ I(f), $ for a general \tef\ $ f $ was first studied by Eremenko \cite{E}, and the fast escaping set, $ A(f), $ was introduced by Bergweiler and Hinkkanen in \cite{BH99}.  The set $ A(f) $ has stronger properties than $ I(f) $, and it now plays an important role in transcendental dynamics.  For a comprehensive treatment of $ A(f) $, including many new results on its properties and references to previous work, we refer to Rippon and Stallard \cite{RS10a}.

One fruitful innovation in \cite{RS10a} is the notion of the \textit{levels} of the fast escaping set. For $ L \in \Z $, the $ L $th level of $ A(f) $ with respect to $ R $ is the set
\[ A_R^L(f) = \lbrace z \in \C : \vert f^n(z) \vert \geq M^{n+L}(R, f), \text{ for } n \in \N, n + L \geq 0 \rbrace, \]   
and, in particular, we define
\[ A_R(f) = A_R^0(f) = \lbrace z \in \C : \vert f^n(z) \vert \geq M^n(R, f), \text{ for } n \in \N \rbrace. \]
Working with the levels of $ A(f) $ leads both to simplified proofs of results obtained previously, and to deeper insights into the structure of $ A(f). $
 
In \cite{RS10a}, Rippon and Stallard define a set $ E $ to be an \textit{(infinite) spider's web} if $ E $ is connected and there exists a sequence $ (G_n) $ of bounded, \sconn\ domains such that:
\begin{itemize}
\item $ G_{n+1} \supset G_n $, for $ n \in \N $;
\item $ \partial G_n \subset E $, for $ n \in \N $, and
\item $ \bigcup_{n \in \N} G_n = \C. $
\end{itemize}
In \cite[Theorem 1.4]{RS10a}, they show that, if $ A_R(f)^c $ has a bounded component, then each of $ A_R(f), A(f) $ and $ I(f) $ is a \sw.  

This \sw\ form of the escaping set differs significantly from the Cantor bouquet structure observed in the escaping sets of many \tef s in the Eremenko-Lyubich class $ \mathcal{B} $ (i.e. functions whose critical and asymptotic values lie in a bounded set).

It transpires that $ A_R(f) $ is a \sw\ for a wide range of \tef s. It was proved in \cite{RS05} that this is the case whenever $ f $ has a \mconn\ component of $ F(f) $, and in \cite[Theorem 1.9]{RS10a} that there are many classes of functions that do not have such a \mconn\ Fatou component but for which $ A_R(f) $ is a \sw. Other examples of functions for which $ A_R(f) $ is a \sw\ are given by Mihaljevi\'{c}-Brandt and Peter \cite{MP1}, and by Sixsmith \cite{S1}.   

When $ A_R(f) $ is a \sw, many strong dynamical properties hold.  For example, in \cite[Theorem 1.6]{RS10a}, it is shown that, if $ A_R(f) $ is a \sw, then:
\begin{itemize}
\item every component of $ A(f)^c $ is compact, and
\item every point of $ J(f) $ is the limit of a sequence of points, each of which lies in a distinct component of $ A(f)^c. $
\end{itemize}

In this paper, we explore further the properties and dynamical behaviour of the components of $ A(f)^c $ when $ A_R(f) $ is a \sw.  We show that, in this situation, the $ A(f) $ \sw\ has an intricate structure, and that, by adapting known results about the components of $ J(f) $  when $ f $ has a \mconn\ Fatou component, we can obtain new results about the components of $ A(f)^c $ for the wider class of functions where $ A_R(f) $ is a \sw.

The remainder of this introduction explains the organisation of the paper, and states the main results.  

In Section \ref{prelim}, we set out some background material.  We summarise the basic properties of $ A_R(f) $ spiders' webs and of the levels of $ A(f) $, as described in \cite{RS10a}.  These properties will be used frequently throughout the paper.  We also prove a number of preliminary results for later use.

In Section \ref{nature}, we prove the following topological properties of the components of $ A(f)^c $ when $ A_R(f) $ is a \sw.  The definitions of buried points and components, and the meaning of `surrounding', are given in Section \ref{prelim}.  

\begin{theorem} 
\label{buried}
Let $ f $ be a \tef, let $ R>0 $ be such that $ M(r, f) > r $ for $ r \geq R $ and let $ A_R(f) $ be a \sw.  Let $ K $ be a component of $ A(f)^c $.  Then:  
\begin{enumerate} [(a)]
\item $ \partial K \subset J(f) $ and $ \text{int } K \subset F(f) $.  In particular, $ \overline{U} \subset K $ for every Fatou component $ U $ for which $ K \cap \overline{U} \neq \emptyset. $
\item Every \nhd\ of $ K $ contains a closed subset of $ A(f) \cap J(f) $ surrounding~$ K $.  If $ K $ has empty interior, then $ K $ consists of buried points of $ J(f) $.
\item If $ f $ has a \mconn\ Fatou component, then every \nhd\ of $ K $ contains a \mconn\ Fatou component surrounding $ K $.  If, in addition, $ K $ has empty interior, then $ K $ is a buried component of $ J(f) $.
\end{enumerate}
\end{theorem}

Note that, if $ A_R(f) $ is a \sw, then $ f $ maps any component $ K $ of $ A(f)^c $ onto another such component (see Theorem \ref{components} in Section \ref{prelim}).  We call the sequence of iterates of $ K $ its \textit{orbit}, and any infinite subsequence of its iterates a \textit{suborbit}.  If $ f^p(K) = K $ for some $ p \in \N $, then we say that $ K $ is a \textit{periodic component} of $ A(f)^c $.  If $ f^m(K) \neq f^n(K) $ for all $ m > n \geq 0 $, then we say that $ K $ is a \textit{wandering component} of $ A(f)^c. $

In Section \ref{orbit}, we give a characterisation of the orbits of the components of $ A(f)^c $ when $ A_R(f) $ is a \sw.  To do this, we show how we can use a natural partition of the plane to associate with each component of $ A(f)^c $ a unique `itinerary' that captures information about its orbit.  This then enables us to prove:

\begin{theorem}
\label{uncountable}
Let $ f $ be a \tef, let $ R>0 $ be such that $ M(r, f) > r $ for $ r \geq R $, and let $ A_R(f) $ be a \sw.  Then $ A(f)^c $ has uncountably many components:
\begin{myindentpar}{0.5cm}
\begin{enumerate}[(a)]
\item whose orbits are bounded;
\item whose orbits are unbounded but contain a bounded suborbit; and
\item whose orbits escape to infinity.
\end{enumerate}
\end{myindentpar}
\end{theorem}

The set of buried points of $ J(f) $ is called the \textit{residual Julia set} (see \cite{DF} for the properties of this set).   Since there are only countably many Fatou components, we have the following corollary of Theorem \ref{buried}(b) and Theorem \ref{uncountable}.

\begin{theorem}
\label{residual}
Let $ f $ be a \tef, let $ R>0 $ be such that $ M(r, f) > r $ for $ r \geq R $, and let $ A_R(f) $ be a \sw. Then the residual Julia set of $ f $ is not empty.
\end{theorem}

In Section \ref{bounded}, we restrict our attention to those components of $ A(f)^c $ whose orbits are bounded. The proof of our main result, Theorem \ref{poly}, uses a technique similar to that adopted by Kisaka in \cite{K3} and by Zheng in \cite{Z1}; we describe Kisaka's and Zheng's results in Section \ref{bounded}.   

\begin{theorem}
\label{poly}
Let $ f $ be a \tef, let $ R>0 $ be such that $ M(r, f) > r $ for $ r \geq R $, and let $ A_R(f) $ be a \sw. Let $ K $ be a component of $ A(f)^c $ whose orbit is bounded. Then there exists a polynomial $ g $ of degree at least $ 2 $ such that each component of $ A(f)^c $ in the orbit of $ K $ is quasiconformally homeomorphic to a component of the filled Julia set of $ g $.
\end{theorem}

The existence of the quasiconformal mapping in Theorem \ref{poly} enables us to use recent results from polynomial dynamics \cite{QY, RY1, RY2} to say more about the nature of the components of $ A(f)^c $ whose orbits are bounded:

\begin{theorem}
\label{coroll}
Let $ f $ be a \tef, let $ R>0 $ be such that $ M(r, f) > r $ for $ r \geq R $, and let $ A_R(f) $ be a \sw.  
\begin{enumerate}[(a)]
\item Let $ K $ be a component of $ A(f)^c $ with bounded orbit. Then:
\begin{enumerate}[(i)]
\item $ K $ is a singleton if and only if its orbit includes no periodic component of $ A(f)^c $ containing a critical point.  In particular, if $ K $ is a wandering component of $ A(f)^c $, then $ K $ is a singleton.
\item If the interior of $ K $ is non-empty, then this interior consists of non-wandering Fatou components.  If these Fatou components are not Siegel discs, then they are Jordan domains.
\end{enumerate}
\item All but countably many of the components of $ A(f)^c $ with bounded orbits are singletons.
\end{enumerate}  
\end{theorem}

Note that, by Theorem \ref{coroll}(a)(i), if all of the critical points of $ f $ have unbounded orbits (for example, if they all lie in $ I(f) $), then every component of $ A(f)^c $ with bounded orbit is a singleton.

Evidently, periodic components of $ A(f)^c $ have bounded orbits, so Theorems \ref{poly} and \ref{coroll} apply to them in particular.  Our final section, Section \ref{empty}, gives a further result for periodic components of $ A(f)^c $.   

Dom\'{i}nguez \cite{PD1} has shown that, if $ f $ is a \tef\ with a \mconn\ Fatou component, then $ J(f) $ has buried singleton components, and such components are dense in $ J(f) $ (see also \cite{DF}). Bergweiler \cite{wB00} has given an alternative proof of this result, using a method involving the construction of a singleton component of $ J(f) $ which is also a repelling periodic point of $ f $.  

By using a method similar to Bergweiler's, together with earlier results from this paper, we are able to prove the following.

\begin{theorem}
\label{singleton}
Let $ f $ be a \tef, let $ R>0 $ be such that $ M(r, f) > r $ for $ r \geq R $, and let $ A_R(f) $ be a \sw.  Then $ A(f)^c $ has singleton periodic components, and such components are dense in $ J(f). $  If $ f $ has a \mconn\ Fatou component, then these singleton periodic components of $ A(f)^c $ are buried components of $ J(f). $   
\end{theorem}

Note that, if $ f $ is a \tef\ with a \mconn\ Fatou component, then we have shown that singleton \textit{periodic} components of $ J(f) $ are dense in $ J(f) $, a slight strengthening of the results of Dom\'{i}nguez \cite{PD1} and Bergweiler \cite{wB00}.

The first part of Theorem \ref{singleton} is also a strengthening of Rippon and Stallard's result \cite[Theorem 1.6]{RS10a} that, if $ A_R(f) $ is a \sw, then every point in $ J(f) $ is the limit of a sequence of points, each of which lies in a distinct component of $ A(f)^c $.  Note that, by Theorem \ref{buried}(b), if $ A_R(f) $ is a \sw, then any singleton component of $ A(f)^c $ must be a buried point of $ J(f) $, but if $ f $ does not have a \mconn\ Fatou component, then such a component of $ A(f)^c $ is not a buried component of $ J(f) $, because $ J(f) $ is a \sw\ by \cite[Theorem 1.5]{RS10a} and so is connected.   

\textbf{Acknowledgements} I thank my doctoral supervisors, Prof P.J. Rippon and Prof G.M. Stallard, for their inspiration, and for their particular help and encouragement in the preparation of this paper.

\section[prelim]{Preliminary material}
\label{prelim}
\setcounter{equation}{0} 

We first summarise a number of basic results and definitions that are used throughout this paper.  These are taken from \cite{RS10a}, which should be consulted for full details and proofs.  

First, from the definition of $ A(f) $ and its levels, we have:
\begin{equation} \label{union}
A(f) = \bigcup_{L\in\N} A_R^{-L}(f),
\end{equation}
and
\begin{equation} \label{nested}
f(A_R^L(f)) \subset A_R^{L+1}(f) \subset A_R^L(f), \text{ for } L \in \Z.
\end{equation}
These relations easily give that $ A(f) $ is completely invariant.

Some basic properties of $ A_R(f) $ spiders' webs are given in the following:

\begin{lemma} [Lemma 7.1(a)-(c) in \cite{RS10a}] \label{basic}
Let $ f $ be a \tef, let $ R>0 $ be such that $ M(r,f)>r $ for $ r \geq R $ and let $ L \in \Z $.
\begin{enumerate}[(a)]
\item If $ G $ is a bounded component of $ A_R^L(f)^c $, then $ \partial G \subset A_R^L(f) $ and $ f^n $ is a proper map of $ G $ onto a bounded component of $ A_R^{n+L}(f)^c $, for each $ n \in \N $.
\item If $ A_R^L(f)^c $ has a bounded component, then $ A_R^L(f) $ is a \sw\ and hence every component of $ A_R^L(f)^c $ is bounded.
\item $ A_R(f) $ is a \sw\ if and only if $ A_R^L(f) $ is a \sw.
\end{enumerate}
\end{lemma}

Next, we give some notation and terminology.  In this paper, if $ S $ is a subset of $ \C $, we use the notation $ \widetilde{S} $ to denote the union of $ S $ and all its bounded complementary components (if any).  As in \cite{RS10a}, we say that $ S $ \textit{surrounds} a set or a point if that set or point lies in a bounded complementary component of $ S $.  If $ S $ is a bounded domain and $ f $ is an entire function, then we have:  
\begin{equation} \label{tilde}
f( \widetilde{S} ) \subset \widetilde{f(S)},
\end{equation} 
since if $ \gamma $ is any Jordan curve in $ S $, then the image under $ f $ of the inside of $ \gamma $ lies inside $ f(\gamma) $, and so in $ \widetilde{f(S)} $.

We also recall the following definition:
\begin{definition*}[Definition 7.1 in \cite{RS10a}]
Let $ f $ be a \tef\ and let $ R>0 $ be such that $ M(r,f)>r $ for $ r \geq R $.  If $ A_R(f) $ is a \sw\ then, for each $ n \geq 0 $, let:
\begin{itemize}
\item $ H_n $ denote the component of $ A_R^n(f)^c $ containing $ 0 $, and
\item $ L_n $ denote its boundary, $ \partial H_n $.
\end{itemize}
We say that $ (H_n)_{n \geq 0} $ is the \textit{sequence of fundamental holes} for $ A_R(f) $ and $ (L_n)_{n \geq 0} $ is the \textit{sequence of fundamental loops} for $ A_R(f) $.
\end{definition*}  
Note that $ L_n $ may have bounded complementary components other than $ H_n $.  

The following lemma gives some properties of these sequences.

\begin{lemma} [Lemma 7.2 in \cite{RS10a}] \label{props}
Let $ f $ be a \tef\ and let $ R>0 $ be such that $ M(r,f)>r $ for $ r \geq R $.  Suppose that $ A_R(f) $ is a \sw, and that $ (H_n)_{n \geq 0} $ and $ (L_n)_{n \geq 0} $ are respectively the sequences of fundamental holes and loops for $ A_R(f) $.  Then:
\begin{enumerate}[(a)]
\item $ H_n \supset \lbrace z : \vert z \vert < M^n(R) \rbrace $ and $ L_n \subset A_R^n(f) $, for $ n \geq 0 $;  
\item $ H_{n+1} \supset H_n, $ for $ n \geq 0 $;
\item for $ n \in \N $ and $ m \geq 0 $,
\[ f^n(H_m) = H_{m+n} \text{  and  } f^n(L_m) = L_{m+n}; \]
\item there exists $ N \in \N $ such that, for $ n \geq N $ and $ m \geq 0 $,
\[ L_{n+m} \cap L_m = \emptyset; \]
\item if $ L \in \Z $ and $ G $ is a component of $ A_R^L(f)^c $, then:
\[ f^n(G) = H_{n+L} \text{  and  } f^n(\partial G) = L_{n+L}, \]
for $ n $ sufficiently large;
\item if there are no multiply connected Fatou components, then $ L_n \subset J(f) $ for $ n \geq 0. $
\end{enumerate}
\end{lemma}

We also include in this section a number of other results which will be used in proving the theorems stated in Section \ref{intro}.  

First, we will need the following characterisation of \mconn\ Fatou components for a \tef, due to Baker \cite{iB84}.
\begin{lemma} \label{baker}
Let $ f $ be a \tef\ and let $ U $ be a \mconn\ Fatou component.  Then:
\begin{itemize}
\item $ f^k(U) $ is bounded for any $ k \in \N $,
\item $ f^{k+1} (U) $ surrounds $ f^k(U) $ for large $ k $, and
\item $ f^k(U) \to \infty $ as $ k \to \infty. $
\end{itemize}
\end{lemma}

Next, we will make use of the following topological characterisation of the buried components of a closed set.  We will use the result only where the closed set is the Julia set of a \tef, but we present it in a more general form to bring out its essentially topological nature.  The result may be known, but we have been unable to locate a reference so we include a proof for completeness.  

Recall that, if $ K $ is a component of a closed set $ F $ in $ \widehat{\C} $, then:
\begin{itemize}
\item $ z \in K $ is a \textit{buried point} of $ F $ if $ z $ does not lie on the boundary of any component of $ F^c $, and 
\item $ K $ is a \textit{buried component} of $ F $ if $ K $ consists entirely of buried points of~$ F $. 
\end{itemize}  
In particular, a buried point of $ J(f) $ is a point of $ J(f) $ that does not lie on the boundary of any Fatou component, and a buried component of $ J(f) $ is a component of $ J(f) $ consisting entirely of such buried points.

\begin{theorem}
\label{top}
Let $ K $ be a component of a closed set $ F $ in $ \widehat{\C} $.  Then $ K $ is a buried component of $ F $ if and only if, for each component $ L $ of $ K^c $, and any closed subset $ B $ of $ L $, there is a component of $ F^c $ that separates $ B $ from $ K $ and whose boundary does not meet $ K. $
\end{theorem}

\begin{proof}
Let $ K $ be a buried component of $ F $, let $ L $ be a component of $ K^c $ and let $ B $ be any closed subset of $ L $.   Then $ X = B \cup F $ is closed in $ \widehat{\C} $, $ K $ is a component of $ X $ and $ B $ lies in some other component of $ X $, say $ X'.$

Then it follows from \cite[Theorem 3.3, p.143]{New} that there is a Jordan curve $ C $ separating $ K $ from $ X' $ in such a way that $ C $ lies in $ X^c \subset F^c $.  Since $ C $ is connected, it must lie in some component $ G $ of $ F^c $.  Furthermore, by \cite[Theorem 14.5, p.124]{New}, the complementary component of $ G $ containing $ K $ contains exactly one component ($ D, $ say) of $ \partial G $.  Since $ K $ is a buried component of $ F $, we therefore have $ D \cap K = \emptyset $, as required.

To prove the converse, let $ K $ be a component of $ F $.  Suppose there exists some component $ G $ of $ F^c $ and some $ z \in K $ such that $ z \in \partial G $.  Let $ L $ be the component of $ K^c $ containing $ G $, and let $ B $ be a closed subset of $ G $.  Now suppose that there is a component $ G' $ of $ F^c $ separating $ B $ from $ K $ (and hence $ B $ from $ z $), whose boundary does not meet $ K $.  Then since $ B \subset G $ and $ z \in \partial G $, $ G' $ must meet $ G $.  But $ G $ is a component of $ F^c $, so this means that $ G' = G $, which is a contradiction because $ \partial G \cap K \neq \emptyset. $
\end{proof}

Finally, we will need the following result on mappings of the components of $ A(f)^c $.

\begin{theorem} \label{components}
Let $ f $ be a \tef\ and let $ R>0 $ be such that $ M(r,f)>r $ for $ r \geq R $. If $ A_R(f) $ is a \sw, and $ K $ is a component of $ A(f)^c $, then $ f(K) $ is also a component of $ A(f)^c.$
\end{theorem}

\begin{proof}
As $ A(f) $ is completely invariant, it is clear that $ f(K) $ must lie in a component of $ A(f)^c $, say $ K'. $

Since $ A_R(f) $ is a \sw, components of $ A(f)^c $ are compact \cite[Theorem~1.6]{RS10a}, so each component of $ f^{-1}(K') $ must be closed and lie in some component of $ A(f)^c $.  One such component must contain $ K $, and indeed be equal to $ K $ since $ K $ is itself a component of $ A(f)^c $.  

Suppose $ w \in K' \setminus f(K). $  Since $ A_R(f) $ is a \sw, there exists a bounded, simply connected domain $ G $ containing $ K $ whose boundary lies in $ A(f) $.  The domain $ G $ can contain only a finite number of components of $ f^{-1}(K'). $

Now by \cite[Theorem 3.3, p.143]{New}, there is a Jordan curve $ C $ lying in $ G $ that surrounds $ K $ and separates $ K $ from all other components of $ f^{-1}(K'). $   It follows that $ f(C) $ is a curve that surrounds $ f(K) $ and does not meet $ K' $.  Furthermore, $ f(C) $ cannot surround $ w \in K' $ since $ C $ does not surround any solution of $ f(z) =~w $.  This contradicts the connectedness of $ K' $, and it follows that $ K' \setminus f(K) = \emptyset $.  Thus $ f(K) $ is a component of $ A(f)^c $, as required.  
\end{proof}

\section{The topology of components of $ A(f)^c $}
\label{nature}

\setcounter{equation}{0} 

In this section, we prove Theorem \ref{buried}.  Throughout the section, let $ f $ be a \tef, let $ A_R(f) $ be a \sw\ and let $ K $ be a component of $ A(f)^c $. 

Since $ J(f) = \partial A(f) $ \cite[Theorem 5.1]{RS10a}, it is immediate that $ \partial K \subset J(f) $ and $ \textit{int } K \subset F(f) $. If $ K $ meets the closure of some Fatou component $ U $, then we must have $ U \cap A(f) = \emptyset $, since otherwise $ \overline{U} \subset A(f) $ by \cite[Theorem 1.2]{RS10a}.  Hence $ U \subset K $.  But since $ A_R(f) $ is a \sw, $ K $ is compact, so $ \overline{U} \subset K. $  This proves part (a).  

For the proof of parts (b) and (c), observe that, using (\ref{union}) and (\ref{nested}), we can write:
\begin{equation*} 
K = \bigcap_{l \in \N} G_l, \qquad \text{     with } G_l \supset G_{l+1} \text{   for all } l \in \N, 
\end{equation*}
where $ G_l $ is the component of $ A_R^{-l} (f)^c $ containing $ K $. Thus, for any \nhd\ $ V $ of $ K $, there exists $ M \in \N $ such that $ \overline{G}_l \subset V $ for all $ l \geq M. $  

Now let $ (H_n)_{n \geq 0} $ and $ (L_n)_{n \geq 0} $ be the sequences of fundamental holes and loops for $ A_R(f) $, as defined in Section \ref{prelim}.  Then it follows from Lemma \ref{props}(e) that, for any $ m \geq 0 $ and for sufficiently large $ n \geq m $:
\begin{equation*}
 f^{M + n} (\partial G_{M+m}) = L_{n-m}.
\end{equation*}
Now, for any $ n \geq 0 $, $ L_n \subset A(f) $ and, if $ f $ has no \mconn\ Fatou components, we also have $ L_n \subset J(f) $ by Lemma \ref{props}(f).  Since $ L_n $ is closed, it follows that $ V $ contains a closed subset of $ A(f) \cap J(f) $ that surrounds $ K $.  Further, if $ K $ has empty interior, then part (a) implies that $ K $ consists of buried points of $ J(f) $.  This proves part (b) in the case where there are no \mconn\ Fatou components.  

Now suppose that $ f $ has a \mconn\ Fatou component, $ U $.  Note that $ \overline{U} \subset A(f) $, by \cite[Theorem 4.4]{RS10a}.  Thus part (c) of Theorem \ref{buried} implies part (b), and we need only prove part (c).  

Since $ L_0 $ is bounded, Lemma \ref{baker} implies that we can choose $ k \in \N $ so that:
\begin{equation*}
\label{BWD}
f^{j+1}(U) \text{ surrounds } f^j(U) \text{ for } j \geq k, \text{ and } f^k(U) \text{ surrounds } L_0 .
\end{equation*}   
Since $ f^k(U) $ is bounded, it follows from Lemma \ref{props}(a) that we may also choose $ P \in \N $ so that:
\begin{equation*}
\label{surrounds}
f^k(U) \subset H_P.
\end{equation*}   

Furthermore, by Lemma \ref{props}(e), there exists $ N \in \N $ (depending on $ M $ and $ P $) such that: 
\begin{align*}
\label{G}
f^{M + N + P} (G_M) & = H_{N+P},\\
f^{M + N + P} (\partial G_M) & = L_{N+P},\\
f^{M + N + P} (G_{M+P}) & = H_N
\end{align*}
and
\begin{align*}
f^{M + N + P} (\partial G_{M+P}) & = L_N.
\end{align*}

Since $ f^k(U) \subset H_P $, it is clear that $ f^{N + k}(U) \subset H_{N + P} $, and by our choice of $ k $, $ f^{N + k}(U) $ surrounds $ f^k(U) $.  

We claim that $ f^{N + k}(U) $ also surrounds $ H_N $.  For let $ W $ denote the interior of the complementary component of $ f^k(U) $ that contains $ H_0 $.  Then $ W \subset \widetilde{f^k(U)} $, so, using (\ref{tilde}):
\[ H_N = f^N(H_0) \subset f^N(W) \subset f^N(\widetilde{f^k(U)}) \subset \widetilde{f^{N + k}(U)}.\]
Now $ \partial W \subset \partial f^k(U) \subset J(f) $, and thus it follows that $ \partial f^N(W) \subset f^N(\partial W) $ cannot meet $ f^{N + k}(U). $  We have therefore shown that $ f^N(W) $ lies in $ \widetilde{f^{N + k}(U)} $, but that its boundary does not meet $ f^{N + k}(U) $.  Thus $ f^{N + k}(U) $ surrounds $ f^N(W) $ and hence $  H_N $, as claimed.

We now show that $ G_M $ must contain a \mconn\ Fatou component and that this surrounds $ K. $  To do this, let $ \Gamma $ be a Jordan curve in $ f^{N + k}(U) $ that surrounds $ 0 $.  Then, of the finitely many components of $ f^{-(M + N + P)} (\widetilde{\Gamma}) $ that lie in $ G_M $, one must contain $ G_{M+P} $, since $ f^{M + N + P} (G_{M+P}) = H_N \subset \widetilde{\Gamma}. $  Call this component $ \Lambda $, and its boundary $ \gamma $.  Since $ f^{M + N + P} $ is a proper map of the interior of $ \Lambda $ onto the interior of $ \widetilde{\Gamma} $, we have $ f^{M + N + P}(\gamma) = \Gamma $, and thus $ \gamma $ must lie in a Fatou component, $ U' $ say, that is contained in $ G_M $.  Furthermore, $ U' $ is \mconn, since $ \gamma $ surrounds $ G_{M+P} $ which contains $ \partial K \subset J(f). $  Thus $ G_M $ contains a \mconn\ Fatou component surrounding $ K $, and therefore so does our arbitrary \nhd\ $ V $ of $ K $.    

Finally, suppose that $ K $ has empty interior, so $ K \subset J(f) $.  Since $ A(f) $ is connected, $ K^c $ has only one component, and the remainder of part~(c) therefore follows immediately from Theorem \ref{top}.

\section{Orbits of components of $ A(f)^c $}
\label{orbit}

\setcounter{equation}{0} 

In this section, we give a characterisation of the orbits of the components of $ A(f)^c $ when $ A_R(f) $ is a \sw, and prove that $ A(f)^c $ then has uncountably many components of various types (Theorem \ref{uncountable}).  To this end, we first describe a natural partition of the plane that enables us to encode information about the orbits of the components of $ A(f)^c $.  

Throughout this section, let $ f $ be a \tef, let $ R>0 $ be such that $ M(r, f) > r $ for $ r \geq R $, and let $ A_R(f) $ be a \sw.  Recall from Theorem \ref{components} that $ f $ maps a component $ K $ of $ A(f)^c $ onto another such component.  We refer to the sequence of iterates of $ K $ as its \textit{orbit}, and to any infinite subsequence of its iterates as a \textit{suborbit}.

To construct the partition, we proceed as follows.  Let $ (L_m)_{m \geq 0} $ be the sequence of fundamental loops for $ A_R(f) $, as defined in Section \ref{prelim}.  Now, by Lemma \ref{props}(c) and (d), there exists $ N \in \N $ such that $ L_{N+m} \cap L_m = \emptyset $ and $ f^N(L_{mN}) = L_{(m+1)N} $ for $ m \geq 0 $.  Thus $(L_{mN})_{m \geq 0} $ is a sequence of disjoint loops, and $ f^N $ maps any such loop onto its successor in the sequence. We use these loops to define our partition.  To simplify the exposition, we assume (without loss of generality) that $ N = 1 $, so that our sequence of disjoint loops is $(L_m)_{m \geq 0} $.

Now define
\begin{equation} \label{defn1}
B_0 = H_0,  
\end{equation}
and
\begin{equation} \label{defn2}
B_m = H_m \setminus H_{m-1}, \text{  for  } m \geq 1, 
\end{equation}
where $ (H_m)_{m \geq 0} $ is the sequence of fundamental holes for $ A_R(f) $. 

Then, for each $ m \geq 1 $, $ B_m $ is a connected set surrounding $ 0 $ and:
\[ \partial B_m = L_m \cup L_{m-1}. \]
Also, for any $ k \in \N $:
\[ \bigcup_{m \geq k} B_m = \C \setminus H_{k-1},\] 
and indeed $ \bigcup_{m \geq 0} B_m = \C $.  It follows that the sets $ B_m, m \geq 0, $ form a partition of the plane.   

Hence, for each point $ z \in \C $, there is a unique sequence $ \underline{s} = s_0s_1s_2 \ldots $ of non-negative integers (which we call the \textit{itinerary} of $ z $ with respect to $ A_R(f) $), such that:
\[ f^k(z) \in B_{s_k},  \text{ for } k \in \N \cup \lbrace 0 \rbrace. \]

Evidently, the itinerary of a point encodes information about its orbit, and we now investigate which orbits are possible.  We begin with the following lemma, whose proof is based on an argument in the proof of \cite[Lemma 6]{RS09}.

\begin{lemma}
\label{coverB}
Let $ B_m, m \geq 0, $ be as defined in (\ref{defn1}) and (\ref{defn2}).  Then, for each $ m \geq 0 $, exactly one of the following must apply:
\begin{equation} \label{first}
f(\overline{B}_m) = \overline{B}_{m+1},
\end{equation}
or
\begin{equation} \label{second}
f(\overline{B}_m) = \overline{H}_{m+1}.
\end{equation} 
Furthermore, (\ref{second}) holds for $ m = 0 $ and for infinitely many $ m $.
\end{lemma}

\begin{proof}
Note first that, since $ f $  maps compact sets to compact sets and is an open mapping, we have:
\begin{equation} \label{BIIB}
\partial f (\overline{B}_m) \subset f(\partial \overline{B}_m).
\end{equation}
Now if $ m \geq 1 $, then clearly $ \overline{H}_m = H_{m-1} \cup \overline{B}_m, $ and so, by Lemma \ref{props}(c):
\[ \overline{H}_{m+1} = H_m \cup f(\overline{B}_m).  \]
We thus have:
\begin{equation} \label{B}
\overline{B}_{m+1} \subset f(\overline{B}_m) \subset \overline{H}_{m+1}.
\end{equation}
But $ f(\partial \overline{B}_m) = L_{m+1} \cup L_m $, so (\ref{BIIB}) implies that if $ f $ maps any point of $ \overline{B}_m $ into $ H_m $, then $ f(\overline{B}_m) $ must contain the whole of $ H_m. $  Taken together with (\ref{B}), this shows that (\ref{first}) and (\ref{second}) are the only possibilities for $ m \geq 1 $.  Note also that $ f(\overline{B}_0) = \overline{H}_1 $, so (\ref{second}) applies when $ m = 0 $.  

Now suppose that (\ref{second}) held for only finitely many $ m $.  Then, for sufficiently large~$ k $, we would have:
\[ f(\C \setminus H_{k-1}) = \bigcup_{m \geq k} f(\overline{B}_m) \subset \C \setminus H_{k-1}, \]
so $ \C \setminus H_{k-1} $ would lie in the Fatou set, which is impossible.  
\end{proof}

It follows from Lemma \ref{coverB} that there is a strictly increasing sequence of integers $ m(j), j \geq 0 $, with $ m(0) = 0 $, such that (\ref{second}) holds if and only if $ m = m(j) $, for $ j \geq 0 $.  

We need the following lemma \cite[part of Lemma 1]{RS09}:
\begin{lemma} \label{cover}
Let $ E_n, n \geq 0, $ be a sequence of compact sets in $ \C $, and $ f : \C \to \widehat{\C} $ be a continuous function such that
\[ f(E_n) \supset E_{n+1}, \text{   for } n \geq 0. \]
Then there exists $ \zeta $ such that $ f^n(\zeta) \in E_n $, for $ n \geq 0 $.
\end{lemma} 

We now describe a rule for constructing integer sequences, $ \underline{s} = s_0s_1s_2 \ldots $, such that:
\begin{itemize}
\item the itinerary of any point $ z \in \C $ satisfies the rule, and
\item with limited exceptions, any integer sequence constructed according to the rule corresponds to the itinerary of some point $ z \in \C. $
\end{itemize}
Recall that, although we are assuming that $ N = 1 $ for convenience of exposition, these itineraries are defined with respect to iteration under $ f^N $, where $ N \in \N $ is such that
\[ L_{N+m} \cap L_m = \emptyset,  \quad \text{ for } m \geq 0. \]
Our rule for constructing integer sequences $ \underline{s} = s_0s_1s_2 \ldots $ is that, for each $ n \geq 0 $, we derive $ s_{n+1} $ from $ s_n $ as follows:
\begin{enumerate}[(1)]
\item if $ s_n = m(j) $ for some $ j \geq 0 $, then:
\[ s_{n+1} \in \lbrace 0, 1, 2, \ldots, m(j), m(j) + 1 \rbrace;  \]
\item otherwise, $ s_{n+1} = s_n + 1. $
\end{enumerate}

The itinerary of any point $ z \in \C $ satisfies this rule by Lemma \ref{coverB}, since:
\begin{equation}
\label{loops}
f(L_m) = L_{m+1}, \quad  m \geq 0.
\end{equation} 
On the other hand, if $ \underline{s} $ is an integer sequence constructed according to this~rule, and we put $ E_n = \overline{B}_{s_n} $ for $ n \geq 0 $, then it follows from Lemma \ref{coverB} that the sequence of compact sets $ (E_n)_{n \geq 0} $ and the function $ f $ satisfy the conditions of Lemma \ref{cover}.  Hence there exists a point $ z \in E_0 = \overline{B}_{s_0} $ such that $ f^n(z) \in E_n = \overline{B}_{s_n} $, for $ n \geq 0 $.

But the itineraries of points are defined relative to the sets $ B_m, m \geq 0 $, which partition the plane, rather than the compact sets $ \overline{B}_m = B_m \cup L_m $ we have used in the construction of points corresponding to  integer sequences.  However, if any iterate of a point lies in $ L_m $ for some $ m \geq 0 $, then all subsequent iterates also lie in a fundamental loop by (\ref{loops}).  Thus the only situation in which an integer sequence constructed according to our rule may not coincide with the itinerary of a point derived from the sequence by using Lemma \ref{cover} is where the orbit of the point ends on the fundamental loops $ (L_m)_{m \geq 0}. $ 

In particular, since $ L_m \subset A(f), m \geq 0 $, if an integer sequence $ \underline{s} $ gives rise to a point $ z $ in $ A(f)^c $, then the itinerary of $ z $ is  $ \underline{s} $.  Furthermore,  any two points in $ A(f)^c $ with different itineraries must necessarily lie in different components of $ A(f)^c $, so all points in the same component of $ A(f)^c $ as $ z $ have itinerary $ \underline{s} $.   

We are now in a position to prove Theorem \ref{uncountable} which states that, if $ A_R(f) $ is a \sw, then $ A(f)^c $ has uncountably many components:
\begin{myindentpar}{0.5cm}
\begin{enumerate}[(a)]
\item whose orbits are bounded;
\item whose orbits are unbounded but contain a bounded suborbit; and
\item whose orbits escape to infinity.
\end{enumerate}
\end{myindentpar}

\begin{proof} [Proof of Theorem \ref{uncountable}]

We examine each of the orbit types $ (a) - (c) $ in turn, showing how to construct an itinerary for a point in a component of $ A(f)^c $ with that type of orbit, and proving that there must be uncountably many such components.  Note that many alternative constructions are possible for each orbit type.

For type $ (a) $, components with bounded orbit, we can construct an itinerary in the following way:
\begin{itemize}
\item choose $ j_0 \geq 2 $, and put $ s_0 = m(j_0) $;
\item for $ n \geq 0 $:
\begin{enumerate}[(i)]
\item if $ s_n = m(j_0) $, put $ s_{n+1} = m(j_0)-1 $; 
\item otherwise, put $ s_{n+1} = s_n + 1 $. 
\end{enumerate} 
\end{itemize}

Evidently, by Lemma \ref{cover} and the ensuing discussion, we thereby obtain a point $ a \in B_{m(j_0)} \cap A(f)^c $ whose orbit is bounded.

To prove that there are uncountably many such points, we use an idea from a proof by Milnor \cite[Corollary 4.15, p.49]{Mil}. Given any finite partial itinerary $ s_0s_1 \ldots s_k $ corresponding to the first $ k $ iterations of the point $ a $, then for the next value of $ n > k $ for which $ s_n = m(j_0) $, instead of assigning $ s_{n+1} $ the value  $ m(j_0) - 1 $ under (i) above, we could instead put $ s_{n+1} = m(j_0) - 2. $ The remaining $ s_n $ are then chosen as above.  By Lemma \ref{cover}, this sequence gives rise to another point $ a' \in B_{m(j_0)} \cap A(f)^c $ with the same finite partial itinerary $ s_0s_1 \ldots s_k $ as~$ a $, but with an ultimately different bounded orbit.  

Thus the finite partial itinerary $ s_0s_1 \ldots s_k $ can be extended in two different ways to yield two further finite partial itineraries, each of which may again be extended in the same way. By continuing this process, it follows that $ s_0s_1 \ldots s_k $ can be extended in uncountably many ways, and Lemma \ref{cover} shows that each resulting infinite itinerary corresponds to a distinct point in $ B_{m(j_0)} \cap A(f)^c $.  Since any two points in $ A(f)^c $ with different itineraries must lie in different components of $ A(f)^c, $ it follows that there are uncountably many components of $ A(f)^c $ with bounded orbits.

To construct an itinerary of type $ (b) $, i.e. for a component of $ A(f)^c $ whose orbit is unbounded but contains a bounded suborbit, we can proceed as follows:
\begin{itemize}
\item put $ s_0 = 0 $;
\item for $ n \geq 0 $:
\begin{enumerate}[(i)]
\item if there exists $ j \geq 2 $ such that $ s_n = m(j) $ and $  s_i \neq m(j) $ for $ i = 0, 1, 2, \ldots, n-1 $, put $ s_{n+1} = 0 $; 
\item otherwise, put $ s_{n+1} = s_n + 1. $
\end{enumerate}
\end{itemize}

By Lemma \ref{cover}, we thereby obtain a point $ b \in B_0 \cap A(f)^c $ whose orbit is unbounded, but which visits $ B_0 $ infinitely often.  Evidently, at any stage when the orbit returns to $ B_0 $, we could equally well have returned it to $ B_1 $, and it therefore follows by the same argument as for type $ (a) $ that there are uncountably many components of $ A(f)^c $ with orbits of type $ (b) $.

Finally, consider type $ (c) $, i.e. components of $ A(f)^c $ whose orbits escape to infinity.  

For each $ i \in \N $, let $ j_i $ be the largest value of $ j $ such that:
\[ \widetilde{B}_{m(j)} \subset \lbrace z : \vert z \vert < M^i(R) \rbrace, \]
or, if no such values of $ j $ exist, let $ j_i = 0. $  Let $ I $ be the smallest value of $ i $ for which $ j_i \neq 0. $

To construct an itinerary of type $ (c) $, our procedure is:
\begin{itemize}
\item put $ s_0 = m(j_I) $;
\item for $ n \geq 0 $:
\begin{enumerate}[(i)]
\item if $ s_n = m(j_i) $ for some $ i \geq I $, and if $  n \leq 2i - I $, put $ s_{n+1} = m(j_i) $; 
\item otherwise, put $ s_{n+1} = s_n + 1. $
\end{enumerate}
\end{itemize}

The purpose of this construction is to keep the orbit of the constructed point within the closure of $ \widetilde{B}_{m(j_i)}, i \geq I, $ until at least $ 2i - I $ iterations have taken place.  To see that a point $ z $ with such an itinerary lies in $ A(f)^c $, note that, for all $ i \geq I $:
\begin{equation}
\label{twice}
\vert f^{2i - I}(z) \vert < M^i(R).
\end{equation}
It follows that there is no value of $ L \in \N $ such that:
\[ \vert f^{i+L}(z) \vert \geq M^i(R), \text{  for all } i \in \N, \]
since putting $ i = I + L $ contradicts (\ref{twice}).  Thus, from the definition, $ z \notin A(f). $

It now follows from Lemma \ref{cover} that we obtain a point $ c \in B_{m(j_I)} \cap A(f)^c $ which escapes to infinity.  Given any finite partial itinerary $ s_0s_1 \ldots s_k $ corresponding to the first $ k $ iterations of $ c $, then for the next value of $ n > k $ such that, for some $ i \geq I $,
\[ s_n = m(j_i) \text{  and   }  n = 2i - I + 1, \]
instead of applying (ii) above, we could equally well put $ s_{n+1} = m(j_i) $. The finite partial itinerary $ s_0s_1 \ldots s_k $ can therefore be extended in two different ways to yield two further finite partial itineraries, corresponding to two different points in $ B_{m(j_I)} \cap A(f)^c $ with the same initial iteration sequence, but with ultimately different orbits escaping to infinity. Thus, using the same argument as previously, there are uncountably many components of $ A(f)^c $ with orbits of type $ (c) $.  This completes the proof.
\end{proof}

\textit{Remark.}  The method of proof of Theorem \ref{uncountable} can also be applied to show the existence of components of $ A(f)^c $ with other types of orbits.  For example, using \cite[Theorem 1]{RS09}, we can adapt the proof for orbits of type (c)  to show that, if $ A_R(f) $ is a \sw, then there are uncountably many components $ K $ of $ A(f)^c $ whose orbits escape to infinity arbitrarily slowly, in the sense that, if $ (a_n) $ is any positive sequence such that $ a_n \to \infty $ as $ n \to \infty $, then:
\[ \vert f^n(z) \vert \leq a_n, \quad \text{ for sufficiently large } n, \text{ and for all } z \in K. \]

\section{Components of $ A(f)^c $ with bounded orbits}
\label{bounded}
\setcounter{equation}{0} 

In this section, we again assume that $ A_R(f) $ is a \sw, and we examine further the components of $ A(f)^c $ with bounded orbits.  We show that, in this case, we can say much more about the nature of such components than is given by Theorem \ref{buried}.  We do this by following a method used by Kisaka \cite{K3} and Zheng \cite{Z1}. 

Building on results in \cite{K2}, Kisaka proved in  \cite[Theorem A]{K3} that, if $ f $ is a \tef\ with a \mconn\ Fatou component, and  $ C $ is a component of $ J(f) $ with bounded orbit, then there is a polynomial $ g $ such that $ C $ is quasiconformally homeomorphic to a component of the Julia set of $ g $.  Furthermore, he showed that: 
\begin{enumerate}
\item if the complement of $ C $ is connected, then $ C $ is a buried component of $ J(f) $;
\item if $ C $ is a wandering component of $ J(f) $, then it is a buried singleton component of $ J(f) $.
\end{enumerate}

Zheng \cite{Z1} used a similar technique to obtain results about Fatou components with unbounded orbits for certain \tef s (see our remark following the proof of Theorem \ref{coroll} below for further details). 

We now prove results analogous to those of Kisaka, but expressed in terms of components of $ A(f)^c $ rather than of $ J(f) $, and with $ f $ belonging to the wider class of \tef s for which $ A_R(f) $ is a \sw. 

Note that Theorem \ref{buried}(b) and (c) already gives us an analogue of (1) in Kisaka's result.  Indeed, it does  more, for there we do not assume that the component of $ A(f)^c $ has bounded orbit.

We now prove Theorem \ref{poly}, which establishes the existence of a quasiconformal conjugacy with a polynomial for components of $ A(f)^c $ with bounded orbit when $ A_R(f) $ is a \sw.  The proof uses the notion of a \textit{polynomial-like map}, introduced by Douady and Hubbard, and their Straightening Theorem (see Chapter VI of \cite{CG}, and \cite{DH}).
\begin{definition*}
Let $ D_1 $ and $ D_2 $ be bounded, \sconn\ domains with smooth boundaries such that $ \overline{D_1} \subset D_2 $.  Let $ h $ be a proper analytic map of $ D_1 $ onto $ D_2 $ with $ d $-fold covering, where $ d \geq 2. $  Then the triple $ (h; D_1, D_2) $ is termed a \textit{polynomial-like map} of degree $ d $.  

We also define the \textit{filled Julia set} $ K(h; D_1, D_2) $ of the polynomial-like map $ h $ to be the set of points all of whose iterates remain in $ D_1 $, i.e.  
\[  K(h; D_1, D_2) = \bigcap_{k \geq 0} h^{-k} (D_1).\]
\end{definition*}

\begin{theorem*}[Douady and Hubbard, Straightening Theorem]
If $ (h; D_1, D_2) $ is a polynomial-like map of degree $ d \geq 2 $, then there exists a quasiconformal map $ \phi : \C \to \C $ and a polynomial $ P $ of degree $ d $ such that $ \phi \circ h = P \circ \phi $ on $ \overline{D_1} $.  Moreover:
\[ \phi ( K(h; D_1, D_2)) = K(P),  \]
where $ K(P) $ is the filled Julia set of the polynomial $ P. $
\end{theorem*}

\begin{proof}[Proof of Theorem \ref{poly}]

Let $ K $ be a component of $ A(f)^c $ with bounded orbit, and let the sequences of fundamental holes and loops for $ A_R(f) $ be $ (H_n)_{n \geq 0} $ and $ (L_n)_{n \geq 0} $ respectively. 

Since $ f $ is transcendental and the orbit of $ K $ is bounded, it follows from Lemma~\ref{props} parts (a) and (c) that we may choose $ m \in \N $ so large that the orbit of $ K $ lies in $ H_m $, and such that $ f $ is a proper map of $ H_m $ onto $ H_{m+1} $ of degree at least $ 2 $.  It then follows from Lemma \ref{props}(d) that there exists $ N \in \N $ such that:
\begin{itemize}
\item $ f^N $ is a proper map of $ H_m $ onto $ H_{m+N} $ and of $ H_{m+N} $ onto $ H_{m+2N} $, of degree at least $ 2 $; and
\item $ L_{m + N} \cap L_m = \emptyset $ and $ L_{m + 2N} \cap L_{m + N} = \emptyset $.
\end{itemize}

Now let $ \gamma $ be a smooth Jordan curve in $ H_{m + 2N} $ that surrounds $ \overline{H}_{m+N} $ and does not meet any of the critical values of $ f^N $, and let $ V $ be the bounded component of $ \gamma^c $, so that:
\[ H_m \subset \overline{H}_{m+N} \subset V \subset H_{m+2N}. \]

Define $ U $ to be the component of $ f^{-N}(V) $ that contains $ H_m $ (and hence the orbit of $ K $). Then $ U $ must lie in the component of $ f^{-N}(H_{m+2N}) $ that contains $ H_m $, i.e. $ U \subset H_{m+N} $, and so we have $ \overline{U} \subset V. $   Then $ U $ is simply connected, and $ f^N : U \to V $ is a proper map of degree at least $ 2 $.  Furthermore, since $ V $ is bounded by a smooth Jordan curve that does not meet any of the critical values of $ f^N $, it follows that $ U $ is also bounded by a smooth Jordan curve.   We have therefore established that the triple $ (f^N; U, V) $ is a polynomial-like map of degree at least $ 2 $.

Now the set $ \overline{U} $ consists of a collection of components (or parts of components) of $ A(f)^c $, together with a bounded subset of $ A(f). $  Clearly points in $ A(f) $ cannot lie in the filled Julia set $ K(f^N; U, V) $, but points in $ A(f)^c $ may do so. In particular, since the orbit of the component $ K $ under iteration by $ f $ lies in $ U $, it must also lie in $ K(f^N; U, V) $. 

Indeed, since $ f $ maps every component of $ A(f)^c $ onto another such component (Theorem \ref{components}), and points in $ A(f) $ cannot lie in $ K(f^N; U, V) $, then every component of $ A(f)^c $ in the orbit of $ K $ must be a distinct component of $ K(f^N; U, V). $  Now, by the Straightening Theorem, there is a polynomial $ g $ of degree at least~$ 2 $ such that $ K(f^N; U, V) $ is quasiconformally homeomorphic to the filled Julia set of $ g $, and thus it follows that each component of $ A(f)^c $ in the orbit of $ K $ is quasiconformally homeomorphic to a component of the filled Julia set of $ g $.
\end{proof}  

The existence of the quasiconformal mapping in Theorem \ref{poly} enables us to use polynomial dynamics to draw some further conclusions, and in particular to prove Theorem \ref{coroll}.  Part~(a)(i) of Theorem  \ref{coroll} is an analogue of $ (2) $ in Kisaka's result.  We will use the following in our proof.  

\begin{theorem*} [Qiu and Yin, Main Theorem in \cite{QY}]
For a polynomial $ g $ of degree at least $ 2 $, a component of the filled Julia set of $ g $ is a singleton if and only if its forward orbit includes no periodic component containing a critical point. 
\end{theorem*}

\begin{theorem*} [Roesch and Yin \cite{RY1, RY2}]
If $ g $ is a polynomial of degree at least $ 2 $, then any bounded Fatou component which is not a Siegel disc is a Jordan domain. 
\end{theorem*}

\begin{proof}[Proof of Theorem \ref{coroll}]
Let $ K $ be a component of $ A(f)^c $ with bounded orbit.  Then by Theorem \ref{poly}, any component of $ A(f)^c $ in the orbit of $ K $ is quasiconformally homeomorphic to a component of the filled Julia set of some polynomial $ g $ of degree at least $ 2 $. Since periodic orbits and critical points are preserved by the homeomorphism, it follows from the above result of Qiu and Yin that $ K $ is a singleton if and only if its orbit includes no periodic component of $ A(f)^c $ containing a critical point.  Wandering components of $ A(f)^c $ clearly have no periodic components in their orbit, so if $ K $ is a wandering component, then it must be a singleton. This proves part (a)(i).

To prove part (a)(ii), let $ K $ again denote a component of $ A(f)^c $ with bounded orbit.  If the interior of $ K $ is non-empty, then by Theorem \ref{buried} the interior must consist of one or more components of $ F(f). $ Now the quasiconformal homeomorphism obtained in the proof of Theorem \ref{poly} maps the interior of a component of $ A(f)^c $ onto the interior of a component of the filled Julia set of a polynomial $ g $ of degree at least $ 2 $, which consists of Fatou components of $ g $ that must be non-wandering.  Part (a)(ii) now follows immediately from the above result of Roesch and Yin, since Siegel discs and Jordan curves are clearly preserved by the homeomorphism.

Part (b) follows from part (a)(i) because $ f $ has only countably many critical points. 
\end{proof}

\textit{Remark.}  In \cite[Theorem 3]{Z1}, Zheng used a method similar to that adopted in the proof of Theorem \ref{poly} to show that, if $ f $ is a \tef\ with a \mconn\ Fatou component, and if $ U $ is any wandering Fatou component, then there exists a subsequence $ f^{n_k} $ of $ (f^n)_{n \in \N} $ such that $ f^{n_k} \vert_U \to \infty $ as $ k \to \infty $.  Thus, the orbit of every wandering Fatou component is unbounded.  

Note that it follows from our Theorem \ref{coroll}(a)(ii) that the orbit of every wandering Fatou component is unbounded whenever $ A_R(f) $ is a \sw.

In \cite[Theorem 4]{Z1}, Zheng used the same method as in his Theorem 3 to show that every wandering Fatou component has an unbounded orbit for \tef s such that
\begin{equation}
\label{min}
m(r, f) := \min{ \lbrace \vert f(z) \vert: \vert z \vert = r \rbrace} > r, \text{  for an unbounded sequence of } r.   
\end{equation}
Zheng's proof of this result could readily be adapted to show that, if there is a sequence of bounded, \sconn\ domains $ D_n $ with smooth boundaries such that:
\begin{itemize}
\item $ \bigcup_{n \in \N} D_n = \C $,
\item $ \overline{D}_n \subset D_{n+1} $, for $ n \in \N, $ and
\item $ f(\partial D_n) $ surrounds $ D_{n+1} $ for $ n \in \N, $
\end{itemize}
then the orbit of every wandering Fatou component is unbounded. This result would then cover those \tef s for which (\ref{min}) holds, as well as those for which $ A_R(f) $ is a \sw\ (since for such functions the existence of a sequence of domains $ D_n $ with the above properties follows from our proof of Theorem \ref{poly}).

\section{Periodic components of $ A(f)^c $} 
\label{empty}
\setcounter{equation}{0} 

In this section, we prove Theorem \ref{singleton}, which states that, if $ A_R(f) $ is a \sw, then $ A(f)^c $ has singleton periodic components, and these components are dense in $ J(f). $

Our proof makes use of earlier results from this paper, together with the method used by Bergweiler \cite{wB00} in his alternative proof of the result due to Dom\'{i}nguez~\cite{PD1} stated in Section \ref{intro}.  We use the following corollary of the Ahlfors five islands theorem, proved in \cite{wB00} for a wide class of meromorphic functions, but here stated for \tef s since this is all we need:

\begin{proposition*} \label{islands}
Let $ f $ be a \tef, and let $ D_1, \ldots, D_5 \subset \widehat{\C} $ be Jordan domains with pairwise disjoint closures.  Let $ V_1, \ldots, V_5 $ be domains satisfying $ V_j \cap J(f) \neq \emptyset $ and $ V_j \subset D_j $ for $ j \in \lbrace 1, \ldots, 5 \rbrace. $ Then there exist $ \mu \in \lbrace 1, \ldots, 5 \rbrace $, $ n \in \N $ and a domain $ U \subset V_\mu $ such that $ f^n : U \rightarrow D_\mu $ is conformal.
\end{proposition*}  

\begin{proof} [Proof of Theorem~\ref{singleton}]

It follows from Theorems \ref{uncountable}(a) and \ref{coroll}(b) that, if $ A_R(f) $ is a \sw, then there are uncountably many singleton components of $ A(f)^c $, and by Theorem \ref{buried} these lie in $ J(f). $  

Now $ J(f) $ is the closure of the backward orbit $ O^-(z) $ of any non exceptional point $ z \in J(f) $. But since $ f $ is an open mapping and $ A(f) $ is completely invariant, the preimages of singleton components of $ A(f)^c $ are themselves singleton components of $ A(f)^c $, and it therefore follows that singleton components of $ A(f)^c $ are dense in $ J(f)$. 

We now claim that singleton \textit{periodic} components of $ A(f)^c $ are dense in $ J(f). $

To prove this, let $ W $ be any \nhd\ of a point $ w \in J(f). $  Then since $ J(f) $ is perfect, $ W $ contains infinitely many points in $ J(f) $. Thus there exist $ w_j \in J(f), j \in \lbrace 1, \ldots, 5 \rbrace $ and $ \varepsilon > 0 $ such that the Jordan domains $ D_j = B(w_j, \varepsilon) $ have pairwise disjoint closures and lie in $ W $.

Now since singleton components of $ A(f)^c $ are dense in $ J(f)$, for $ j \in \lbrace 1, \ldots, 5 \rbrace $ there exist singleton components $ \lbrace z_j \rbrace $ of $ A(f)^c $ such that $ z_j \in D_j $, and by Theorem \ref{buried}(b), there are closed subsets $ X_j $ of $ A(f) \cap J(f) $ lying in $ D_j $ and surrounding~$ z_j $.  Let $ V_j $ be the bounded complementary component of $ X_j $ containing~$ z_j $.  Then $ \partial V_j \subset A(f) $ and, since $ z_j \in J(f) $, we have that $ V_1, \ldots, V_5 $ are domains satisfying $ V_j \cap J(f) \neq \emptyset $ for  $ j \in \lbrace 1, \ldots, 5 \rbrace $.   It follows that we may apply the above Proposition, obtaining $ \mu \in \lbrace 1, \ldots, 5 \rbrace $, $ n \in \N $ and a domain $ U \subset V_\mu $ such that $ f^n : U \rightarrow D_\mu $ is conformal.  

Now let $ \phi $ be the branch of the inverse function $ f^{-n} $ which maps $ D_\mu $ onto $ U. $  Then $ \phi $ must have a fixed point $ z_0 \in U \subset V_\mu $.  Furthermore, by the Schwarz lemma, this fixed point must be attracting, and because $ \phi(D_\mu) = U $ where $ \overline{U} $ is a compact subset of $ D_\mu $, we have that $ \phi^k(z) \rightarrow z_0 $ as $ k \rightarrow \infty $, uniformly for $ z \in D_\mu $.

Since $ z_0 $ is an attracting fixed point of $ \phi $, it is a repelling fixed point of $ f^n $ and hence a repelling periodic point of $ f $.  Thus $ z_0 $ lies in $ J(f) \cap A(f)^c $.  

Now $ z_0 = \phi^k(z_0) \in \phi^k(V_\mu)$ for all $ k \in \N. $  Furthermore, $ \text{diam } \phi^k(\overline{V_\mu}) \rightarrow 0 $ as $ k \rightarrow \infty $.  It follows that:
\begin{equation}
\label{single}
\bigcap_{k \in \N} \phi^k(\overline{V_\mu}) = \lbrace z_0 \rbrace. 
\end{equation}
Since $ \partial V_\mu $ lies in $ A(f) $, which is completely invariant, and $ \phi $ is conformal, we have $ \partial \phi^k (V_\mu) = \phi^k (\partial V_\mu) \subset A(f) $ for all $ k \in \N $.  But $ \phi^k(\partial V_\mu) $ surrounds $ z_0 $ for all $ k \in \N $, so $ \lbrace z_0 \rbrace $ must be a singleton component of $ A(f)^c $ by (\ref{single}).  

We have therefore shown that, in any \nhd\ of an arbitrary point of $ J(f) $, there is a singleton component of $ A(f)^c $ that is also a repelling periodic point of $ f $.  This proves the claim.

To complete the proof of the theorem, note finally that if $ f $ has a \mconn\ Fatou component, then it follows from Theorem \ref{buried}(c) that the singleton periodic components of $ A(f)^c $ constructed above are buried components of $ J(f). $
\end{proof}

\end{document}